\documentclass[12pt]{article}
\usepackage[T2A]{fontenc}
\usepackage[cp1251]{inputenc}
\usepackage[english]{babel}
\usepackage{amsmath,latexsym,amsthm,amsfonts,tipa}
\usepackage{amssymb,amscd,graphpap}
\newcommand\w{{\omega}}
\newcommand{\B}{{\mathcal B}}

\newcommand{\F}{\mathcal{F}}
\newcommand{\orr}{\overrightarrow}
\newcommand{\olr}{\overleftarrow}
\newcommand{\olrr}{\overleftrightarrow}
\newcommand{\PP}{{\mathcal P}}
\newcommand{\UU}{{\mathcal U}}
\newcommand{\V}{{\mathcal V}}
\newcommand{\kk}{{\kappa}}

\newtheorem{theorem}{Theorem}

\newtheorem{question}{Question}
\theoremstyle{definition}
\newtheorem{corollary}{Corollary}

\begin{document}

\title{Asymptotic structures of cardinals}
\author{O.~Petrenko, I.~Protasov, S.~Slobodianiuk}
\date{}
\maketitle

\begin{abstract}
A ballean is a set $X$ endowed with some family $\F$ of its subsets, called the balls, in such a way that $(X,\F)$  can be considered as an asymptotic counterpart of a uniform topological space. Given a cardinal $\kappa$, we define $\F$ using a natural order structure on $\kappa$. We characterize balleans up to coarse equivalence, give the criterions of metrizability and cellularity, calculate the basic cardinal invariant of these balleans. We conclude the paper with discussion of some special ultrafilters on cardinal balleans.

\

{\bf 2010 AMS Classification}: 54A25, 05A18

\

\end{abstract}

\section{Introduction}

Following \cite{b13} we say that a {\it ball structure} is a triple $\mathcal{B}=(X,P,B)$, where $X$, $P$ are non-empty sets and, for
every $x\in X$ and $\alpha\in P$, $B(x,\alpha)$ is a subset of $X$
which is called a \emph{ball of radius $\alpha$ around $x$}. It is
supposed that $x\in B(x,\alpha)$ for all $x\in X$ and $\alpha\in P$.
The set $X$ is called the {\it support} of $\mathcal{B}$, $P$ is
called the {\it set of radii}. 

Given any $x\in X, A\subseteq X,
\alpha\in P$, we set
$$B^*(x,\alpha)=\{y\in X:x\in B(y,\alpha)\},\
B(A,\alpha)=\bigcup_{a\in A}B(a,\alpha).$$

A ball structure $\mathcal{B}=(X,P,B)$ is called a {\it ballean} if 

\begin{itemize}
\item[(1)] for any $\alpha,\beta\in P$,
there exist $\alpha',\beta'$ such that, for every $x\in X$,
$$B(x,\alpha)\subseteq B^*(x,\alpha'),\ B^*(x,\beta)\subseteq B(x,\beta');$$

\item[(2)] for any $\alpha,\beta\in P$,
there exists $\gamma\in P$ such that, for every $x\in X$,
$$B(B(x,\alpha),\beta)\subseteq B(x,\gamma);$$

\item[(3)] for any $x,y\in X$, there exists $\alpha\in P$ such that $y\in B(x, \alpha)$.
\end{itemize}

A ballean on $X$ can also be determined in terms of entourages of the diagonal $\Delta_X$ of $X\times X$, in this case it is called a coarse structure \cite{b14}. For balleans as counterparts of uniform topological spaces see \cite[Chapter 1]{b13}.

Let $\mathcal{B}=(X,P,B)$, $\mathcal{B'}=(X',P',B')$
be balleans. A mapping $f:X\to X'$ is called a
$\prec$-\emph{mapping} if, for every $\alpha\in P$, there exists
$\alpha'\in P'$ such that, for every $x\in X$,
$f(B(x,\alpha))\subseteq B'(f(x),\alpha')$. If there exists a bijection $f:X\rightarrow X'$ such that $f$ and $f^{-1}$ are $\prec$-mappings, $\B$ and $\B'$ are called {\it asymorphic} and $f$ is called an {\it asymorphism}.

For a ballean $\B = (X,P,B)$, a subset $Y\subseteq X$ is called {\it large} if there is $\alpha \in P$ such that $X = B(Y, \alpha)$. A subset $V$ of $X$ is called {\it bounded} if $V\subseteq B(x,\alpha)$ for some $x\in X$ and $\alpha\in P$. Each non-empty subset $Y\subseteq X$ determines a {\it subballean} $\B_Y = (Y,P,B_Y)$ where $B_Y(y,\alpha) = Y \cap B(y, \alpha)$.

We say that $\B$ and $\B'$ are {\it coarsely equivalent} if there exist large subset $Y\subseteq X$ and $Y' \subseteq X'$ such that the subballeans $\B_Y$ and $\B'_{Y'}$ are asymorphic.

Given a cardinal $\kappa > 0$ and ordinals $x,\alpha \in \kappa$, we put
$$\orr{B}(x,\alpha) = [x, x + \alpha] = \{y\in \kappa: x\leqslant y \leqslant x + \alpha\}, \olr{B}(x,\alpha) = \{y\in \kappa: x\in[y, y+\alpha]\},$$
$$\olrr{B}(x,\alpha) = \orr{B}(x,\alpha) \cup \olr{B}(x,\alpha),$$
so we have got three ball structures
$$\orr{\kappa} = (\kappa, \kappa, \orr{B}), \olr{\kappa} = (\kappa, \kappa, \olr{B}), \olrr{\kappa} = (\kappa, \kappa, \olrr{B}).$$

It is easy to see that, for $\kappa > 1$, the ball structures $\orr{\kappa}$ and $\olr{\kappa}$ do not satisfy (1), so $\orr{\kappa}$ and $\olr{\kappa}$ are not balleans, but $\olrr{\kappa}$ is a ballean for each $\kappa > 0$. This ballean $\olrr{\kappa}$ is called a {\it cardinal ballean} on $\kappa$.

If $\kappa$ is finite then the ballean $\olrr{\kappa}$ is bounded and any two bounded balleans are coarsely equivalent, so {\bf in what follows all cardinal balleans are supposed to be infinite}.

In section 2, we characterize cardinal balleans up to coarse equivalence. The criterions of metrizability and cellularity are given in section 3. In section 4 we study the basic cardinal invariants of cardinal ballean. In section 5 we consider the action of the combinatorial derivation on cardinals. We conclude the paper with discussion in section 6 of some special ultrafilters on cardinal balleans.

\section{Coarse equivalence}

For a ballean $\B = (X,P,B)$, we use a preordering $<$ on $P$ defined by the rule: $\alpha < \beta$ if and only if $B(x, \alpha) \subseteq B(x,\beta)$ for each $x\in X$. A {\it cofinality} $cf \B$ is the minimal cardinality of cofinal subsets of $P$.

A ballean $\B$ is called {\it ordinal} if there exists a well-ordered by $<$ cofinal subset of $P$. If $\B$ is ordinal then $\B$ is asymorphic to $\B' = (X, \kappa, B')$ for some cardinal $\kappa$. Each ballean coarsely equivalent to some ordinal ballean is ordinal. Clearly, $\olrr{\kappa}$ is ordinal for each $\kappa$.

We say that a family $\F$ of subsets of $X$ is {\it uniformly bounded in $\B$} if there exists $\alpha\in P$ such that, for every $F\in \F$, there is $x\in X$ such that $F\subseteq B(x,\alpha)$.

\begin{theorem}
A ballean $\B(X, \kappa, B)$ is coarsely equivalent to $\olrr{\kappa}$ if and only if there exists $x_0\in X$ and $\gamma\in\kappa$ such that
\begin{itemize}
\item[(i)] $B(x_0, \alpha + \gamma)\setminus B(x_0, \alpha) \ne \varnothing$ for each $\alpha \in \kappa$;
\item[(ii)] for every $\beta \in \kappa$, the family $\{B(x_0, \alpha + \beta)\setminus B(x_0, \alpha) : \alpha\in \kappa\}$ is uniformly bounded in $\B$.
\end{itemize}
\end{theorem}

\begin{proof}
Clearly $\olrr{\kappa}$ satisfies $(i)$ and $(ii)$ with $x_0 = 0$ and $\gamma = 1$.

To prove the converse statement, we choose a subset $A\subseteq \kappa$ such that $[\alpha, \alpha + \gamma) \cap [\alpha' , \alpha' + \gamma) = \varnothing$ for all distinct $\alpha, \alpha' \in A$ and $\kappa = \cup_{\alpha \in A}[\alpha, \alpha + \gamma)$. For each $\alpha \in A$, we use $(i)$ to pick some element $x(\alpha) \in B(x_0, \alpha + \gamma) \setminus B(x_0, \alpha)$ and note that the set $L = \{x(\alpha): \alpha \in A\}$ is large in $\B$. Then we denote by $f:A\to \kappa$ the natural well-ordering of $A$ as a subset of $\kappa$ and, for each $\alpha \in A$, put $h(x(\alpha)) = f(\alpha)$. Clearly, $h$ is a bijection from $L$ to $\kappa$. By $(ii)$, $h$ is an asymorphism between $L$ and $\kappa$. Since $L$ is large in $\B$, we conclude that $\B$ and $\kappa$ are coarsely equivalent.
\end{proof}

For distinct cardinals $\kappa$ and $\kappa'$, the balleans $\olrr{\kappa}$ and $\olrr{\kappa'}$ are not coarsely equivalent because (see Theorem \ref{t_4} $(i)$) each large subset of $\olrr{\kappa}$ (resp. $\olrr{\kappa'}$ has cardinality $\kappa$ (resp. $\kappa'$), so there are no bijections (in particular, asymorphisms) between large subsets of $\olrr{\kappa}$ and $\olrr{\kappa'}$.

\section{Metrizability and cellularity}

Each metric space $(X,d)$ defines a {\it metric ballean} $(X, \mathbb{R}^+, B_d)$, where $B_d(x,r) = \{y \in X: d(x,y) \leqslant r\}$. By \cite[Theorem 2.1.1]{b13}, for a ballean $\B$, the following conditions are equivalent: $cf \B \leqslant \aleph_0$, $\B$ is asymorphic to some metric ballean, $\B$ is coarsely equivalent to some metric ballean. In this case $\B$ is called {\it metrizable}. Applying this criterion, we get

\begin{theorem}
For a cardinal $\kappa$, the following conditions are equivalent:
\begin{itemize}
\item[(i)] $cf \kappa = \aleph_0$;
\item[(ii)]  $\olrr{\kappa}$ is asymorphic to some metric ballean;
\item[(iii)] $\olrr{\kappa}$ is coarsely equivalent to some metric ballean.
\end{itemize}
\end{theorem}

Given an arbitrary ballean $\B = (X,P,B)$, $x,y\in X$ and $\alpha\in P$, we say that $x$ and $y$ are {\it $\alpha$-path connected} if there exists a finite sequence $x_0, \ldots, x_n$, $x_0 = x, x_n = y$ such that $x_{i+1}\in B(x_i, \alpha)$ for each $i\in \{0, \ldots , n-1\}$. For any $x\in X$ and $\alpha\in P$, we set

$$B^\square(x, \alpha) = \{y\in X: x,y \mbox{ are $\alpha$-path connected}\}.$$

The ballean $\B^\square = (X, P, B^\square)$ is called a {\it cellularization} of $\B$. A ballean $\B$ is called {\it cellular} if the identity mapping $id:X\to X$ is an asymorphism between $\B$ and $\B^\square$. By \cite[Theorem 3.1.3]{b13}, $\B$ is cellular if and only if $asdim\,\B = 0$, where $asdim\,\B$ is the asymptotic dimension of $\B$. By \cite[Theorem 3.1.1]{b13}, a metric ballean is cellular if and only if $\B$ is asymorphic to a ballean of some ultrametric space.

\begin{theorem}
For a cardinal $\kappa$, the ballean $\olrr{\kappa}$ is cellular if and only if $\kappa > \aleph_0$.
\end{theorem}

\begin{proof}
The ballean $\olrr{\aleph_0}$ is not cellular because $B^\square(0,1) = \aleph_0$ so $\olrr{\aleph_0}^\square$ is bounded but $\olrr{\aleph_0}$ is unbounded.

Assume that $\kappa > \aleph_0$, fix an arbitrary $\alpha \in \kappa$ and note that 

$$B^\square(x, \alpha) \subseteq \bigcup_{n\in \aleph_0} \olrr{B}(x, n\alpha)\subseteq \olrr{B}(x,\gamma),$$
where $\gamma = \sup_{n\in \aleph_0} n\alpha$. Since $\kappa > \aleph_0$, we have $\gamma \in \kappa$. Hence, $id:\kappa \to \kappa$ is an asymorphism between $\olrr{\kappa}$ and $\olrr{\kappa}^\square$.
\end{proof}

\section{Cardinal invariants}

Given a ballean $\mathcal{B}=(X,P,B)$, a subset $A$ of $X$ is called 
\begin{itemize}
\item {\it large} if $X=B(A,\alpha)$ for some $\alpha\in P$;
\item {\it small} if $X\setminus B(A,\alpha)$ is large for every $\alpha \in P$;
\item {\it thick} if, for every $\alpha \in P$, there exists $a\in A$ such that $B(a,\alpha)\subseteq A$;
\item {\it thin} if, for every $\alpha\in P$, there exists a bounded subset $V$ of $X$ such that $B(a,\alpha)\cap B(a',\alpha) = \varnothing$ for all distinct $a, a'\in A\setminus V$.
\end{itemize}

We note that large, small, thick and thin subsets can be considered as asymptotic counterparts of dense, nowhere dense, open and uniformly discrete subsets subsets of a uniform topological space.

The following cardinal characteristics of a ballean $\B$ were introduced and studied in \cite{b1}, \cite{b9}, \cite{b10}, see also \cite[Chapter 9]{b13}. All these characteristics are invariant under asymorphisms.

\begin{itemize}
\item[] $den \,\mathcal{B} = \min\{|L|$: $L$ is a large subset of $X\},$

\item[] $res\,\mathcal{B} = \sup\{|\mathcal{F}|: \mathcal{F}\mbox{ is a family of pairwise disjoint
large subsets of }X\},$

\item[] $cores\,\mathcal{B} = \min\{|\mathcal{F}|: \mathcal{F}\mbox{ is a covering of
}X\mbox{ by small subsets}\},$

\item[] $thick\,\mathcal{B} = \sup\{|\mathcal{F}|: \mathcal{F}\mbox{ is a family of pairwise disjoint
thick subsets of }X\},$

\item[] $thin\,\mathcal{B} = \min\{|\mathcal{F}|: \mathcal{F}\mbox{ is a covering of
}X\mbox{ by thin subsets}\},$

\item[] $spread\, \mathcal{B} = \sup\{|Y|_{\mathcal{B}}$: $Y$ is a thin subset of $X\},$ where $|Y|_{\mathcal{B}}=\min\{|Y\setminus V|: V$ is a bounded subset of $X.\}$
\end{itemize}

To prove Theorem \ref{t_4}, we need some ordinal arithmetics (see \cite{b2}). An ordinal $\gamma$, $\gamma \ne 0$ is called {\it additively indecomposable} if one of the following statements holds

\begin{itemize}
\item $\alpha + \beta < \gamma$ for any ordinal  $\alpha, \beta < \gamma$;
\item  $\alpha + \gamma < \gamma$ for every ordinal  $\alpha < \gamma$;
\end{itemize}

Every ordinal $\alpha > 0$ can be written uniquiely in the Cantor normal form
$$\alpha = n_1 \cdot \gamma_1 + \ldots + n_k \cdot \gamma_k,$$
where $\gamma_1 > \ldots > \gamma_k$ are additively indecomposable, $n_1, \ldots, n_k$ are natural numbers. We put 

$$|\alpha| = |\{\beta: \beta < \alpha\}|, ||\alpha|| = n_1 + \ldots + n_k, r(\alpha) = \gamma_k.$$

\begin{theorem}\label{t_4}
For every cardinal $\kappa$, the following statements hold
\begin{itemize}
\item[(i)] $den \olrr{\kappa} = spread \olrr{\kappa} = \kappa$;
\item[(ii)] $res \olrr{\kappa} = \kappa$ and $\olrr{\kappa}$ can be partitioned in $\kappa$ large subsets;
\item[(iii)] $cores \olrr{\kappa} = \aleph_0$;
\item[(iv)] $thick \olrr{\kappa} = \kappa$ and $\olrr{\kappa}$ can be partitioned in $\kappa$ thick subsets;
\item[(v)] $thin \olrr{\kappa} = cf \kappa$ if $\kappa$ is a limit cardinal and $thin\, \kappa^+ = \kappa$.
\end{itemize}
\end{theorem}

\begin{proof}
$(i)$ For a large subset $L$ of $X$, we pick $\alpha \in P$ such that $X = B(L, \alpha)$, observe that $|B(x, \alpha)|\leqslant 2|\alpha|$ so $|X| \leqslant |L|(2|\alpha|)$ and $|X| = |L|$. Since the ballean $\olrr{\kappa}$ is ordinal, by \cite[Theorem 2.3]{b1}, $spread \B = den \B$ and there is a thin subset $Y$ of $X$ such that $|Y|_\B = |Y|$.

$(ii)$ For $\kappa = \aleph_0$, we identify $\kappa\setminus \{0\}$ with the set $\mathbb{N}$ of natural numbers and partition $\mathbb{N} =\cup_{n\in \aleph_0} 2^nN $, where $N$ is the set of all odd numbers. Since each subset $2^nN$ is large in $\aleph_0$, we get a desired statement.

Assume that $\kappa > \aleph_0$, we have $|L| = \kappa$, so it suffices to show that each subset $L(\gamma)$ is large in $\kappa$. Then $\olrr{B}(\alpha, \gamma)\cap L(\gamma) \ne \varnothing$ for each $\alpha \in \kappa$.

$(iii)$ The statement is trivial for $\kappa = \aleph_0$ because each singleton in $\kappa$ is small.

Assume that $\kappa > \aleph_0$ and, for each $n\in \mathbb{N}$, put

$$S_n = \{\alpha \in \kappa: ||\alpha|| = n\}.$$
We show that, for all $n\in \mathbb{N}$ and $\gamma \in \kappa$, $X\setminus \olrr{B}(S_n, \gamma)$ is large so $S_n$ is small. Since $\kappa > \aleph_0$, we have $|\Gamma| = \kappa|$ and $\Gamma$ is cofinal in $\kappa$. Thus, we may suppose that $\gamma \in \Gamma$. We take any $x\in S_n$, $y\in \olrr{B}(x,\gamma)$ and note that if $y$ has a member $m\cdot \gamma$ in its Cantor normal form then $m\leqslant n + 1$. On the other hand, if $k\cdot \gamma$ is the last member of $x + (n+2)\cdot \gamma$ in its Cantor normal form then $k \geqslant n + 2$. Therefore $\olrr{B}(x, (n+2)\cdot \gamma) \cap (\kappa \setminus B(S_n,\gamma)) \ne \varnothing$. It follows that $\kappa \setminus \olrr{B}(S_n, \gamma)$ is large.

$(iv)$ Since the ballean $\olrr{\kappa}$  is ordinal, by \cite[Theorem 3.1]{b10}, $thick \olrr{\kappa} = den \olrr{\kappa}$ and there is a disjoint family of cardinality $den \olrr{\kappa}$ consisting of thick subsets, so we can apply $(ii)$.

$(v)$ Assume that $\gamma^+ < \kappa$ but $\kappa$ can be partitioned into $\gamma$ thin subsets $\kappa = \cup_{\beta < \gamma} T_\beta$. By the definition of thin subsets, for each $\beta < \gamma$, there is $x(\beta) \in \kappa$ such that $|[\alpha, \alpha + \gamma^+]\cap T_\beta|\leqslant 1$ for each $\alpha > x(\beta)$. We choose $x\in \kappa$ such that $x > x(\beta)$ for each $\beta < \gamma$. Then $|[x, x + \gamma^+]\cap \cup_{\beta < \gamma} T_\beta|\leqslant \gamma$. Since $|[x, x + \gamma^+]| = \gamma^+$, we have $|[x, x + \gamma^+] \setminus \cup_{\beta < \gamma} T_\beta \ne \varnothing$, contradicting $\kappa = \cup_{\beta < \gamma} T_\beta$.

If $\kappa$ is a limit cardinal, by above paragraph, $\kappa$ can not be partitioned into $< cf \kappa$ thin subsets. Since each bounded subset is thin, we have $thin\, \kappa = cf \kappa$.

To conclude the proof, we partition $\kappa^+$ into $\kappa$ thin subsets. Let $\{\gamma_\alpha: \alpha < \kappa^+\}$ be the numeration of all additively indecomposable ordinals from $[\kappa, \kappa^+)$. For $\alpha \in \kappa^+$, we fix some bijection $f_\alpha:\gamma \to [\gamma_\alpha, \gamma_{\alpha + 1})$ and, for each $\lambda \in \kappa$, put
$$T_\lambda = \{f_\alpha(\lambda): \alpha \in \kappa^+\}.$$
Then $[\kappa, \kappa^+) = \cup_{\lambda < \kappa} T_\lambda$. Clearly $[0, \kappa)$ is bounded (and so thin) and if $x\in T_\lambda$ and $x > \gamma_\alpha$ then $\olrr{B}(x, \gamma_\alpha) \cap T_\lambda = \{x\}$, so each subset $T_\lambda$ is thin.
\end{proof}

It is worth to be marked that a set $\Gamma$ of all additively indecomposable ordinals is not only small ($\Gamma = S_1$) but also thin in $\olrr{\kappa}$, and so "very small" in asymptotic sense. On the other hand, for $\kappa > \aleph_0$, $\Gamma$ is unbounded and closed in the order topology on $\kappa$, and so "very large" in topological sense.

Let $\B = (X,P,B)$ be a ballean and $\alpha \in P$. Following \cite{b12}, we say that a subset $Y\subseteq X$ has {\it asymptotically isolated balls} if, for every $\beta > \alpha$, there is $y\in Y$ such that $B(y,\beta) \setminus B(y, \alpha) = \varnothing$. If $Y$ has asymptotically isolated balls for some $\alpha \in P$, we say that $Y$ has asymptotically isolated balls.

A ballean $\B$ is called {\it asymptotically scattered} if each unbounded subset of $X$ has asymptotically isolated balls. A subset $Y\subseteq X$ is called asymptotically scattered if the subballean $\B_Y$ is asymptotically scattered.

The {\it scattered number} of $\B$ is the cardinal 

$$scat\, \B = \{\min|\F|: \F \mbox{ is a covering of $X$ by asymptotically scattered subsets} \}.$$

\begin{question}
For each $\kappa$, determine or evaluate $scat \olrr{\kappa}$.
\end{question}

We say that a ballean $\B$ is {\it weakly asymptotically scattered} if, for every unbounded subset $Y$ of $X$, there is $\alpha \in P$ such that, for every $\beta \in P$, there exists $y\in Y$ such that $(B(y, \beta)\setminus B(y, \alpha))\cap X = \varnothing$. A subset $Y$ of $X$ is called weakly asymptotically scattered if the subballean $\B_Y$ is weakly asymptotically scattered. For $\kappa > \aleph_0$, each small subset $S_n$ from the proof of Theorem \ref{t_4} $(iii)$ is weakly asymptotically scattered, so $\olrr{\kappa}$ can be partitioned into $\aleph_0$ weakly asymptotically scattered subsets. We note that for $n\geqslant 2$, $S_n$ is not asymptotically scattered.

Let $\gamma$ be a cardinal. Following \cite{b11}, we say that a ballean $\B = (X,P,B)$ is {\it $\gamma$-extraresolvable} if there exists a family $\F$ of large subsets of $X$ such that $|\F| = \gamma$ and $F\cap F'$ is small whenever $F, F'$ are distinct members of $\F$. The {\it extraresolvability} of $\B$ is the cardinal
$$exres\, \B = \{\sup \gamma : \B \mbox{ is $\gamma$-extraesolvable}\}.$$
By \cite[Theorem 4]{b11}, $exres \olrr{\aleph_0} = \aleph_0$.

\begin{question}
For each $\kappa$, determine or evaluate $exres\,\kappa$.
\end{question}

Under some additional to ZFC assumptions, an existence of $\aleph_1$-Kurepa tree (see \cite[p. 74]{tt} we prove that $exres \olrr{\aleph_1} > \aleph_1$. Recall that a poset $(T,<)$ is a tree if for any $t\in T$ the set $sub(t)=\{x\in T: x<t\}$ is well-ordered. For any $t\in T$, $L(t)$ denotes an ordinal type of $sub(t)$ and $L(T)=\sup\{L(t)+1: t\in T\}$. A tree $T$ is a $\kappa$-tree if $L(T)=\kk$ and $|Lev_\alpha(T)|<\kk$, $\alpha < L(T)$ where $Lev_\alpha(T)=\{t\in T: L(t)=\alpha\}$. A $\kk$-tree $T$ is called $\kk$-Kurepa if there exists $\kk^+$ maximal chains of cardinality $\kk$ in $T$. Now let $T$ be an $\aleph_1$-Kurepa tree. Fix any bijection $f_\alpha: Lev_\alpha(T)\to [\alpha\w;(\alpha+1)\w)$, $\alpha < \aleph_1$, and define $f=\bigcup f_\alpha$, $B_i=f^{-1}(C_i),i<\aleph_2$, where $C_i$ is a maximal chain in $T$. Then every $B_i$ is large and, for any $i\neq j$, $|B_i\cap B_j|<\aleph_1$.

\begin{question} In ZFC, does there exist a cardinal $\kk$ such that $exres\olrr{\kappa}>\kappa$? \end{question}

\section{The Combinatorial derivation}

Given a subset $A$ of $\kappa$, we denote 

$$\Delta(A)=\{\alpha \in \kappa: A + \alpha \cap A \mbox{ is unbounded in }\kappa\},$$
$A+\alpha = \{\beta + \alpha: \beta \in A\}$, and say that $\Delta: \PP_\kappa \to \PP_\kappa$ is the combinatorial derivation. For a group version of $\Delta$ and motivation of this definition see \cite{prot}.

\begin{theorem}\label{t_4_1}
If a subset $A$ of $\kappa$ is not small and $\kappa$ is regular then $\Delta(A)$ is large.
\end{theorem}

\begin{proof}
Since $A$ is not small, there exists $\lambda \in \kappa$ such that $\olrr{B}(A, \lambda)$ is thick. We show that $\kappa = \olrr{B}(\Delta(A), \lambda + \lambda)$. 

We take an arbitrary $\delta \in \kappa$. Since $\olrr{B}(A, \lambda)$ is thick, for each $\alpha \in \kappa$, we can choose $x_\alpha\in A$, $x_\alpha > \alpha$ such that

$$[x_\alpha,x_\alpha+\delta + \lambda + \lambda]\subset \olrr{B}(A, \lambda).$$
Since $x_\alpha+\delta + \lambda\in \olrr{B}(A, \lambda)$, there exists $y_\alpha \in [\delta, \delta + \lambda + \lambda]$ such that $x_\alpha + y_\alpha \in A$. The set $X = \{x_\alpha: \alpha \in \kappa\}$ is cofinal in $\kappa$ and $|[\delta, \delta + \lambda + \lambda]| < \kappa$. Since $\kappa$ is regular, we can choose a cofinal subset $X' \subseteq X$ and $y\in [\delta, \delta + \lambda + \lambda]$ such that $y_\alpha = y$ for each $x_\alpha \in Y$. Hence, $y\in \Delta(A)$ and $\delta \in \olrr{B}(y, \lambda + \lambda)$ so $\kappa = \olrr{B}(\Delta(A), \lambda + \lambda)$.
\end{proof}

\begin{corollary}
For every finite partition of a regular cardinal $\kappa = A_1 \cup \ldots\cup A_n$, there exists $i\in \{1, \ldots, n\}$ such that $\Delta(A_i)$ is large.
\end{corollary}

\begin{proof}
It suffices to note that at least one cell of the partition is not small.
\end{proof}

By Theorem \ref{t_4} $(iii)$, each subset $S_n = \{\alpha \in \kappa : ||\alpha|| = n\}$ of $\kappa$ is small. It is clear that $\Delta(S_n) \subseteq \{0\} \cup S_1 \cup \ldots \cup S_n$. Hence, $\Delta(S_n)$ is small and Corollary does not hold for countable partition even if $\kappa$ is arbitrary large.

\begin{question}
Is Theorem \ref{t_4_1} true for all singular cardinals?
\end{question}

By the definition of $\Delta$, $\Delta(A) =0$ for each thin subset $A$ of $\kappa$.

\begin{theorem}\label{t_4_2}
Let $\kappa$ be a regular cardinal, $A\subseteq \kappa$ and $0\in A$. Then there exist two thin subsets $X,Y$ of $\kappa$ such that $\Delta(X\cup Y) = A$.
\end{theorem}

\begin{proof}
We enumerate $A = \{a_\alpha: \alpha < |A|\}$, put $a_\alpha=0$ for each $\alpha \in \kappa$, $\alpha \geqslant |A|$, and $b_\alpha = \max\{\alpha, a_\alpha\}$ for every $\alpha \in \kappa$.

We denote $x_{00} = 0$. Since $\kappa$ is regular, for each $\alpha \in \kappa$, we can choose inductively an increasing sequence ($x_{\alpha\beta})_{\beta \leqslant \alpha}$ in $\kappa$ such that
\begin{itemize}
\item[(1)] $x_{\alpha 0} > \sup_{\beta < \alpha} x_{\beta\beta} + b_\alpha + b_\alpha$;
\item[(2)] $x_{\alpha\beta+1} > x_{\alpha\beta} + b_\alpha + b_\alpha,  \beta < \alpha$.
\end{itemize}
After $\kappa$ steps, we denote $X = \{x_{\alpha\beta}: \beta \leqslant \alpha < \kappa\}$, $Y = \{x_{\alpha\beta} + a_\beta: \beta \leqslant \alpha < \kappa \}$. By (1) and (2), $X$ and $Y$ are thin and $\Delta(X\cup Y) = A$.
\end{proof}

\begin{question}
Is Theorem \ref{t_4_2} true for all singular cardinals?
\end{question}

For a cardinal $\kappa$, we denote by $\kappa^\#$ the family of all ultrafilters on $\kappa$ whose members are unbounded in $\kappa$. 

We say that a subset $A$ of $\kappa$ is {\it sparse} if, for every $\UU\in \kappa$, the set $\{\alpha \in \kappa: A\in \UU + \alpha\}$ is bounded in $\kappa$, where $\UU + \alpha$ is an ultrafilter with the base $\{U+\alpha: U\in \UU\}$. Clearly, the family of all sparse subsets of $\kappa$ is an ideal in the Boolean algebra $\PP_\kappa$. If $A$ is thin (in particular, $A$ is bounded) then $|\{\alpha \in \kappa: A\in \UU + \alpha\}| \leqslant 1$ so $A$ is sparse.

We use sparse subsets to characterize strongly prime ultrafilters in $\kappa^\#$. We endow $\kappa$ with the discrete topology, denote by $\beta\kappa$ the Stone-\v{C}ech compactification of $\kappa$ and use the universal property of $\beta \kappa$ to extend the addition $+$ from $\kappa$ to $\beta \kappa$ in such a way that, for each $\alpha \in \kappa$, the mapping $x\mapsto x+\alpha: \beta\kappa \to \beta \kappa$ is continuous, and, for each $\UU \in \beta \kappa$, the mapping $x\mapsto \UU + x: \beta \kappa \to \beta\kappa$ is continuous (see \cite[Chapter 4]{b4}. To describe a base for the ultrafilter $\UU + \V$, we take any element $V\in \V$ and, for every $x\in V$, choose some element $U_x\in \UU$. Then $\cup_{x\in V}(U_x + x)\in \UU + \V$, and the family of subsets of this form is a base for $\UU + \V$.

We note that $\kappa^\#$ is a subsemigroup of $\beta\kappa$ and say that an ultrafilter $\UU\in \kappa^\#$ is {\it strongly prime} if $\UU$ is not in the closure of $\kappa^\# + \kappa^\#$. With above definitions, we get the following characterization.

\begin{theorem}
An ultrafilter $\UU \in \kappa^\#$ is strongly prime if and only if some member $U\in \UU$ is sparse in $\kappa$.
\end{theorem}

\begin{question}
What can be said about inclusions between families of sparse and asymptotically scattered subsets of $\kappa$?
\end{question}

\begin{question}
Is a subset $A$ of $\kappa$ sparse provided that $\Delta(A)$ is sparse?
\end{question}

\section{Around $T$-points}

A free ultrafilter $\UU$ on an infinite cardinal $\kappa$ is called {\it uniform} if $|U| = \kappa$ for each $U\in \UU$. We say that a uniform ultrafilter $\UU$ on $\kappa$ is a {\it $T_\kappa$-point} if, for each minimal well-ordering $<$ of $\kappa$, some member of $\UU$ is thin in the ballean $((\kappa, <), (\kappa, <), \olrr{B})$. We begin with discussion of $T_{\aleph_0}$-points.

Let $G$ be a transitive group of permutations of $\kappa$. We consider a ballean $\B(G, \kappa) = (\kappa, [G]^{< \omega}, B)$, where $[G]^{< \omega}$ is the family of all finite subsets of $G$ and $B(x,F) = \{x\} \cup F(x)$ for each $x\in X$ and $F\in [G]^{< \omega}$, $F(x) = \{f(x), f\in F\}$.

A free ultrafilter $\UU$ on $\aleph_0$ is said to be a $T$-point if, for every countable transitive group $G$ of permutations of $\aleph_0$, some member of $\UU$ is thin in $\B(G, \aleph_0)$. By \cite[Theorem 1]{b7}, every $P$-point and every $Q$-point (in $\beta \omega$) are $T$-points. Under CH, there exist $P$- and $Q$-points but it is unknown \cite[Question 25]{b3} whether in each model of ZFC there exists either $P$-point or $Q$-point. By \cite{b5} and \cite{b6}, this is so if $\mathfrak{c} \leqslant \aleph_2$. By \cite[Proposition 4]{b7}, under CH, there exist a $T$-point which is neither $Q$-point nor weak $P$-point. It is an open question \cite[p.~348]{b8} whether $T$-points exist in ZFC without additional assumptions.

\begin{theorem}
Every $T$-point is a $T_{\aleph_0}$-point.
\end{theorem}

\begin{proof}
A ballean $\B = (X,P,B)$ is called {\it uniformly locally finite} if, for each $\alpha \in P$, there exists a natural number $n$ such that $|B(x, \alpha)|\leqslant n$ for every $x\in X$. By \cite[Theorem 9]{b8}, for $X = \aleph_0$ and a free ultrafilter $\UU$ on $X$, the following statements are equivalent
\begin{itemize}
\item[{\it (i)}] $\UU$ is a $T$-point;
\item[{\it (ii)}] for every metrizable uniformity locally finite ballean $\B$ on $X$, some member of $\UU$ is thin in $\B$;
\item[{\it (iii)}]   for every sequence $(\mathcal{F}_n)_{n\in\omega}$ of uniformly bounded coverings of $X$, there exists $U\in\mathcal{U}$ such that, for each $n\in\omega$, $|F\cap U|\leqslant 1$ for all but finitely many $F\in\mathcal{F}_n$. A covering $\mathcal{F}$ of $\omega$ is uniformly bounded if there is $m\in\aleph_0$ such that, for each $x\in X$, $|st(x,\mathcal{F})|\leqslant m$, $st(x,\mathcal{F})=\bigcup\{F\in\mathcal{F}:x\in F\}$.
\end{itemize}

The equivalence $(i)\Leftrightarrow (ii)$ implies that every $T$-point is a $T_{\aleph_0}$-point because each ballean $((\aleph_0, <), (\aleph_0, <), \olrr{B})$ is metrizable and uniformly locally finite.
\end{proof}

\begin{question}
Does there exist a $T_{\aleph_0}$-point but not a $T$-point?
\end{question}

\begin{question}
Does there exist a $T_{\aleph_0}$-point in ZFC?
\end{question}

Let us call a free ultrafilter $\UU$ on $\aleph_0$ to be {\it $S$-point} if, for every minimal well ordering $<$ of $\aleph_0$, some member of $\UU$ is small in the ballean $((\aleph_0, <), (\aleph_0, <), \olrr{B})$.

\begin{question}
Does there exist an $S$-point in ZFC?
\end{question}

\begin{theorem}\label{t_6}
In ZFC, for every uncountable regular cardinal $\kappa$, there exists a $T_\kappa$-point.
\end{theorem}

\begin{proof}
We say that a uniform ultrafilter $\UU$ on $\kappa$ is a {\it club-ultrafilter} if $\UU$ contains a family of all closed unbounded subsets of $\kappa$. By \cite{b16}, $\UU$ is a $Q$-point, i.e. $\UU$ is selective with respect to any partition of $\kk$ into the cells of cardinality $<\kappa$.

Now we fix some club-ultrafilter $\UU$ on $\kappa$ and prove that $\UU$ is a $T_\kappa$-point. Let $<$ be a minimal well-ordering of $\kappa$ defined by some bijection $h:\kappa \to (\kappa, <)$. We partition inductively $(\kappa, <) = \cup_{\alpha \in (\kappa, <)}[x_\alpha, x_\alpha + \gamma_\alpha)$ into consecutive intervals, where $\Gamma = \{\gamma_\alpha: \alpha \in (\kappa, <)\}$ is the set of all additively indecomposable ordinals of $(\kappa, <)$. By above paragraph, some member $U\in \UU$  meets each subset $h^{-1}([x_\alpha, x_\alpha + \gamma_\alpha))$ in at most one point. By the choice of intervals $[x_\alpha, x_\alpha + \gamma_\alpha)$, $U$ is thin in the ballean $((\kappa, <), (\kappa, <), \olrr{B})$.
\end{proof}

\begin{question}
Is Theorem \ref{t_6} true for uncountable singular cardinals?
\end{question}

\vspace{1cm}

Department of Cybernetics

Kyiv University

Volodimirska 64

01033 Kyiv

Ukraine

i.v.protasov@gmail.com 

\vspace{1cm}

Department of Cybernetics

Kyiv University

Volodimirska 64

01033 Kyiv

Ukraine

opetrenko72@gmail.com 

\vspace{1cm}

Department of Mechanics and Mathematics

Kyiv University

Volodimirska 64

01033 Kyiv

Ukraine

slobodianiuk@yandex.ru

\end{document}